\documentclass[13pt]{amsart}

\usepackage{graphicx}
\usepackage{pinlabel}

\newtheorem{theorem}{Theorem}
\newtheorem{lemma}[theorem]{Lemma}

\newtheorem*{1}{Theorem}
\newtheorem*{2}{The weighted matrix-tree Theorem}

\begin{document}

\title{Extremum problems for eigenvalues of discrete Laplace operators}

\author{Ren Guo}

\address{School of Mathematics, University of Minnesota, Minneapolis, MN, 55455, USA}

\email{guoxx170@math.umn.edu}



\begin{abstract} The discrete Laplace operator on a triangulated polyhedral surface is related to geometric properties of the surface. This paper studies extremum problems for eigenvalues of the discrete Laplace operators. Among all triangles, an equilateral triangle has the maximal first positive eigenvalue. Among all cyclic quadrilateral, a square has the maximal first positive eigenvalue. Among all cyclic $n$-gons, a regular one has the minimal value of the sum of all nontrivial eigenvalues and the minimal value of the product of all nontrivial eigenvalues.
\end{abstract}

\maketitle

\section{Introduction}

\subsection{} A polyhedral surface $S$ is a surface obtained by gluing Euclidean triangles. It is associated with a triangulation $T$. We assume that $T$ is simplicial. Suppose $(\Sigma,T)$ is a polyhedral surface so that
$V,E,F$ are sets of all vertices, edges and triangles in $T.$ We identify vertices of $T$ with indices, edges of $T$ with pairs of indices and triangles of $T$ with triples of indices. This means $V=\{1,2,...|V|\}, E=\{ij\ |\ i,j\in V\}$ and $F=\{\triangle ijk\ |\ i,j,k\in V\}.$  
A vector $(f_1,f_2,...,f_{|V|})^t$ indexed by the set of vertices $V$ defines a piecewise-linear function over $(S,T)$ by linear extension. 

The Dirichlet energy of a function $f$ on $S$ is 
$$E_S(f)=\frac12\int_S|\nabla f|^2 dA.$$

When $f$ is obtained by linear extension of $(f_1,f_2,...,f_{|V|})^t$, the Dirichlet energy of $f$ turns out to be
$$E_{(S,T)}(f)=\frac14\sum_{ijk\in F}[\cot\alpha_{jk}^i(f_j-f_k)^2+\cot\alpha_{ki}^j(f_k-f_i)^2+\cot\alpha_{ij}^k(f_i-f_j)^2]$$
where the sum runs over all triangles of $T$ and for a triangle $ijk\in F$, $\alpha_{jk}^i, \alpha_{ki}^j, \alpha_{ij}^k$ are angles opposite to the edges $jk, ki, ij$ respectively. 

Collecting the terms in the sum above according to edges, we obtain
\begin{align}\label{fm:cot}
E_{(S,T)}(f)=\frac14\sum_{ij\in E}w_{ij}(f_i-f_j)^2
\end{align}
where the sum runs over all edges of $T$ and 
$$
w_{ij}=\left \{
\begin{array}{lll}
\frac12 (\cot\alpha_{ij}^k+\cot\alpha_{ij}^l)           &\ \ \mbox{if $ij$ is shared by two triangles $ijk$ and $ijl$} \\
\frac12 (\cot\alpha_{ij}^k)                             &\ \ \mbox{if $ij$ is contained in one triangles $ijk$}
\end{array}
\right.
$$ 

The Dirichlet energy of a piecewise linear function on a polyhedral surface was introduced and the formula (\ref{fm:cot}) was derived by R. J. Duffin \cite{D}, G. Dziuk \cite{Dz} and U. Pinkall \& K. Polthier \cite{PP} in different context. For application of the Dirichlet energy and formula (\ref{fm:cot}) in the characterization of Delaunay triangulations, see \cite{R, G, BS, CXGL}. For interesting application of the Dirichlet energy and formula (\ref{fm:cot}) in computer graphics, see the survey \cite{BKPAL}.

\subsection{} The discrete Laplace operator $L$ can be introduced by rewriting the Dirichlet energy using notation of matrices
$$E_{(S,T)}(f)=\frac12 (f_1,...,f_{|V|})L(f_1,...,f_{|V|})^t$$
where each entry of the matrix $L$ is given as 
$$
L_{ij}=\left \{
\begin{array}{lll}
\sum_{ik\in E} w_{ik}           &\ \ \mbox{if $i=j$} \\
-w_{ij}                         &\ \ \mbox{if $ij\in E$} \\
0                               &\ \ \mbox{otherwise}  
\end{array}
\right.
$$ 

By definition $L$ is positive semi-definite. Its eigenvalues are denoted by 
$$0=\lambda_0 \leq \lambda_1 \leq \lambda_2 \leq ... \leq \lambda_{|V|-1}.$$

The discrete Laplace operator and its eigenvalues are related to the geometric properties of the polyhedral surface $(S,T)$. For example, it is proved in \cite{CXGL} that among all triangulations, the Delaunay triangulation has the minimal eigenvalues. In \cite{GGLZ}, it is shown that a polyhedral metric on a surface is determined up to scaling by its discrete Laplace operator.   

\subsection{} In smooth case, the spectral geometry is to relate geometric properties of a Riemannian manifold to the spectra of the Laplace operator on the manifold. One of the interesting result is the following one due to G. P\'olya. For reference, for example, see \cite{H}, page 50.  

\begin{1}[P\'olya] 
The equilateral triangle has the least first eigenvalue among all triangles of given area. The square has the least first eigenvalue among all quadrilaterals of given area.
\end{1}

It is conjectured that, for $n\geq 5,$ the regular $n$-gon has the least first eigenvalue among all $n$-gons of given area.

\subsection{} In this paper, similar results as P\'olya's theorem are obtained for the discrete Laplace operator.

\begin{theorem}\label{thm:3} Among all triangles, an equilateral triangle has the maximal $\lambda_1$, the minimal $\lambda_2$ and the minimal $\lambda_1+\lambda_2$.
\end{theorem}

A cyclic polygon is a polygon whose vertices are on a common circle. By adding diagonals, a cyclic polygon is decomposed into a union of triangles. For each inner edge of any triangulation of a a cyclic polygon, the weight $w_{ij}$ is zero. Therefore the discrete Laplace operator is independent of the choice of a triangulation of a cyclic polygon. 

\begin{theorem}\label{thm:4}
Among all cyclic quadrilaterals, a square has the maximal $\lambda_1$, the minimal $\lambda_1+\lambda_2+\lambda_3$, the minimal $\lambda_1\lambda_2+\lambda_2\lambda_3+\lambda_3\lambda_1$ and the minimal $\lambda_1\lambda_2\lambda_3.$
\end{theorem}

\begin{theorem}\label{thm:n} For $n\geq 5,$ among all cyclic $n$-gons, a regular $n$-gon has the minimal $\sum_{i=1}^{n-1}\lambda_i$ and the minimal $\prod_{i=1}^{n-1}\lambda_i$.
\end{theorem}

\subsection{Plan of the paper} Theorem \ref{thm:3}, Theorem \ref{thm:4} and Theorem \ref{thm:n} are proved in section 2, section 3 and section 4 respectively.

\section{Triangles}

\subsection{} In this section we prove Theorem \ref{thm:3}. Let $\theta_1,\theta_2,\theta_3$ be the three angles of a triangle. Let $a_i:=\cot\theta_i$ for $i=1,2,3.$ The condition $\theta_1+\theta_2+\theta_3=\pi$ implies that
\begin{align}\label{fm:3}
a_1a_2+a_2a_3+a_3a_1=1.
\end{align}

The discrete Laplace operator is
$$
L_3=\left(
\begin{array}{cccc}
a_1+a_3 & -a_1    &-a_3 \\
-a_1    & a_1+a_2 &-a_2 \\
-a_3    & -a_2    &  a_2+a_3 
\end{array}
\right).
$$

The characteristic polynomial of $L_3$ is 
\begin{align*}
P_3(x)=\det(L_3-x I_3)&=-x^3+2(a_1+a_2+a_3)x^2-3(a_1a_2+a_2a_3+a_3a_1)x\\
                      &=-x^3+2(a_1+a_2+a_3)x^2-3x
\end{align*}
by the equation (\ref{fm:3}).

The eigenvalues of $L_3$ are denoted by $0=\lambda_0\leq \lambda_1 \leq \lambda_2.$

\subsection{} Therefore $\lambda_1+\lambda_2=2(a_1+a_2+a_3).$ We claim that $a_1+a_2+a_3\geq \sqrt{3}$ and the equality holds if and only if $\theta_1=\theta_2=\theta_3=\frac\pi 3.$

Consider $f:=a_1(\theta_1)+a_2(\theta_2)+a_3(\theta_3)$ as a function defined on the domain 
$$\Omega_3:=\{(\theta_1,\theta_2,\theta_3)\ | \  \theta_1+\theta_2+\theta_3=\pi, \theta_i>0, i=1,2,3 \}.$$

To find the absolute minimum of $f$, we apply the method of Lagrange multiplier. Let 
$$F=a_1(\theta_1)+a_2(\theta_2)+a_3(\theta_3)+y(\theta_1+\theta_2+\theta_3-\pi).$$
Since $$\frac{d a_i(\theta_i)}{d \theta_i}=-\frac{1}{\sin^2\theta_i}=-(1+a_i^2),$$ we have
\begin{align*}
0=\frac{\partial F}{\partial \theta_1}&=-(1+a_1^2)+y,\\
0=\frac{\partial F}{\partial \theta_2}&=-(1+a_2^2)+y,\\
0=\frac{\partial F}{\partial \theta_3}&=-(1+a_3^2)+y,\\
0=\frac{\partial F}{\partial y}&=\theta_1+\theta_2+\theta_3-\pi.
\end{align*}
Therefore the function $f$ has the unique critical point $(\theta_1,\theta_2,\theta_3)=(\frac\pi 3, \frac\pi 3, \frac\pi 3).$

Next, we investigate the behavior of the function $f$ when the variable $(\theta_1,\theta_2,\theta_3)$ approaches the boundary of the domain $\Omega_3$. Let $(\theta_1(t),\theta_2(t), \theta_3(t)), t\in [0,\infty),$ be a path in the domain $\Omega_3$. Let $a_i(t)=\cot \theta_1(t)$ for $i=1,2,3.$ 

Without loss of generality, we assume
$$\lim_{t\to \infty} (\theta_1(t),\theta_2(t), \theta_3(t))=(0, s_2, s_3)$$ where $s_2\geq 0, s_3\geq 0$ and $s_2+s_3=\pi.$ Then $\lim_{t\to \infty}a_1(t)=\infty$ and $a_2(t)+a_3(t)>0$ for $t\in[0,\infty)$. Hence $$\lim_{t\to \infty}(a_1(t)+a_2(t)+a_3(t))=\infty.$$

If $\theta_i+\theta_j<\pi,$ then $\cot\theta_i>\cot(\pi-\theta_j)=-\cot\theta_j$. Therefore $a_i+a_j>0.$ 
Hence $2f=(a_1+a_2)+(a_2+a_3)+(a_3+a_1)>0.$ Thus $f$ has the absolute minimum. But the absolute minimum can not be achieved at a point in the boundary of $\Omega_3$. It must be achieved at the unique critical point $(\frac\pi 3, \frac\pi 3, \frac\pi 3).$

This shows that $a_1+a_2+a_3\geq \sqrt{3}$ and the equality holds if and only if $\theta_1=\theta_2=\theta_3=\frac\pi 3.$

\subsection{} Since
$$\lambda_2=a_1+a_2+a_3+\sqrt{(a_1+a_2+a_3)^2-3},$$ we have $\lambda_2\geq \sqrt{3}$ and the equality holds if and only if $\theta_1=\theta_2=\theta_3=\frac\pi3$.

Since $\lambda_1\lambda_2=3,$ we have $\lambda_1\leq \sqrt{3}$ and the equality holds if and only if $\theta_1=\theta_2=\theta_3=\frac\pi3$.

\section{quadrilaterals}

\subsection{} The vertices of a cyclic quadrilateral decompose its circumcircle into four arcs. We assume the radius of the circumcircle is 1 and the lengths of the four arcs are $2\theta_1,2\theta_2,2\theta_3,2\theta_4.$ Let $a_i:=\cot\theta_i$ for $i=1,...,4.$ The condition $\theta_1+\theta_2+\theta_3+\theta_4=\pi$ implies
\begin{align}\label{fm:4}
a_1a_2a_3+a_1a_2a_4+a_1a_3a_4+a_2a_3a_4=a_1+a_2+a_3+a_4.
\end{align}

There are two ways to decompose a cyclic quadrilateral in to a union of two triangles. The two ways produce the same discrete Laplace operator:
 
$$
L_4=\left(
\begin{array}{cccc}
a_1+a_4 & -a_1    &0        & -a_4 \\
-a_1    & a_1+a_2 &-a_2     & 0    \\
0       &-a_2     &a_2+a_3  & -a_3   \\
-a_4    & 0       &-a_3     &  a_3+a_4 
\end{array}
\right).
$$

The characteristic polynomial of $L_4$ is
\begin{align*}
P_4(x)&=x^4-2(a_1+a_2+a_3+a_4)x^3\\
&\hspace{70pt}+(3(a_1a_2+a_2a_3+a_3a_4+a_4a_1)+4(a_1a_3+a_2a_4))x^2\\
&\hspace{140pt}-4(a_1a_2a_3+a_1a_2a_4+a_1a_3a_4+a_2a_3a_4)x\\
&=x^4-2(a_1+a_2+a_3+a_4)x^3\\
&\hspace{70pt}+(3(a_1a_2+a_2a_3+a_3a_4+a_4a_1)+4(a_1a_3+a_2a_4))x^2\\
& \hspace{206pt} -4(a_1+a_2+a_3+a_4)x,
\end{align*}
by the equation (\ref{fm:4}).

\subsection{}  By the similar argument in the case of triangles, we can show that $a_1+a_2+a_3+a_4$ has the unique critical point at 
$(\frac\pi4,\frac\pi4,\frac\pi4,\frac\pi4).$ And we have $2(a_1+a_2+a_3+a_4)=(a_1+a_2)+(a_1+a_3)+(a_3+a_4)+(a_4+a_1)>0.$

Next, we investigate the behavior of  the function $a_1+a_2+a_3+a_4$ when the variable $(\theta_1,\theta_2,\theta_3, \theta_4)$ approaches the boundary of the domain $$\Omega_4=\{(\theta_1,\theta_2,\theta_3, \theta_4)\ | \  \theta_1+\theta_2+\theta_3+\theta_4=\pi, \theta_i>0, i=1,2,3, 4 \}.$$ Let $(\theta_1(t),\theta_2(t), \theta_3(t), \theta_4(t))$, $t\in [0,\infty),$ be a path in the domain $\Omega_4$. Let $a_i(t)=\cot \theta_1(t)$ for $i=1,2,3,4.$ 

Without loss of generality, we assume
$$\lim_{t\to \infty} (\theta_1(t),\theta_2(t), \theta_3(t), \theta_4(t))=(0, s_2, s_3, s_4)$$ where $s_i\geq 0$ for $i=2,3,4$ and $s_2+s_3+s_4=\pi.$ And we can assume that $s_2<\frac\pi2$ and $s_3<\frac\pi2$.
Then $\lim_{t\to \infty} a_1(t)=\infty$, $a_2(t)>0$ when $t$ is sufficiently large and $a_3(t)+a_4(t)>0$ for any $t\in [0,\infty)$. Hence 
$$\lim_{t\to \infty} (a_1(t)+a_2(t)+a_3(t)+a_4(t))=\infty$$.

Therefore $a_1+a_2+a_3+a_4$ achieves its absolute minimum at the unique critical point $(\frac\pi4,\frac\pi4,\frac\pi4,\frac\pi4).$ Hence $a_1+a_2+a_3+a_4\geq 4$ and the equality holds if and only if $\theta_1=\theta_2=\theta_3=\theta_4=\frac\pi4$.

Therefore $\lambda_1+\lambda_2+\lambda_3\geq 8$, $\lambda_1\lambda_2\lambda_3\geq 16$ and the equality holds if and only if $\theta_1=\theta_2=\theta_3=\theta_4=\frac\pi4$.

\subsection{} To verify the statement about $\lambda_1\lambda_2+\lambda_2\lambda_3+\lambda_3\lambda_1$, by the formula of the characteristic polynomial $P_4(x)$, it is enough to show $$3(a_1a_2+a_2a_3+a_3a_4+a_4a_1)+4(a_1a_3+a_2a_4)\geq 20$$ and the equality holds if and only if $\theta_1=\theta_2=\theta_3=\theta_4=\frac\pi4$.

In fact, consider $g:=3(a_1a_2+a_2a_3+a_3a_4+a_4a_1)+4(a_1a_3+a_2a_4)$ as a function defined on the domain $\Omega_4.$

To find the absolute minimum of $g,$ we apply the method of Lagrange multiplier. 
Let $$G=3(a_1a_2+a_2a_3+a_3a_4+a_4a_1)+4(a_1a_3+a_2a_4)+y(\theta_1+\theta_2+\theta_3+\theta_4-\pi).$$
Then 
\begin{align*}
0=\frac{\partial G}{\partial \theta_1}&=-(3a_2+3a_4+4a_3)(1+a_1^2)+y,\\
0=\frac{\partial G}{\partial \theta_2}&=-(3a_1+3a_3+4a_4)(1+a_2^2)+y,\\
0=\frac{\partial G}{\partial \theta_3}&=-(3a_2+3a_4+4a_1)(1+a_3^2)+y,\\
0=\frac{\partial G}{\partial \theta_4}&=-(3a_3+3a_1+4a_2)(1+a_4^2)+y,\\
0=\frac{\partial G}{\partial y}&=\theta_1+\theta_2+\theta_3+\theta_4-\pi.
\end{align*}

The first and the third equation above imply that
$$(3a_2+3a_4+4a_3)(1+a_1^2)=(3a_2+3a_4+4a_1)(1+a_3^2)$$
which is equivalent to
$$(a_1-a_3)(3a_1a_2+3a_1a_4+3a_2a_3+3a_3a_4+4a_1a_3-4)=0.$$

We claim that the second factor is positive, i.e., $3a_1a_2+3a_1a_4+3a_2a_3+3a_3a_4+4a_1a_3>4.$  

In fact, since $\theta_1+\theta_2+\theta_3<\pi,$ then $\cot(\theta_1+\theta_2)>\cot(\pi-\theta_3).$ Then 
$$\frac{a_1a_2-1}{a_1+a_2}>-a_3$$ which is equivalent to 
\begin{align}\label{fm:ine}
a_1a_2+a_2a_3+a_3a_1>1 
\end{align}
since $a_1+a_2>0.$

By the similar reason, $$a_1a_4+a_4a_3+a_3a_1>1.$$

At least one of $a_2$ and $a_4$ is positive. If $a_2>0$, then 
\begin{align*} 
&3a_1a_2+3a_1a_4+3a_2a_3+3a_3a_4+4a_1a_3\\
&\hspace{50pt}=3(a_1a_4+a_4a_3+a_3a_1)+(a_1a_2+a_2a_3+a_3a_1)+2(a_1+a_3)a_2\\
&\hspace{50pt}>3+1+0.
\end{align*}

If $a_4>0$, then
\begin{align*} 
&3a_1a_2+3a_1a_4+3a_2a_3+3a_3a_4+4a_1a_3\\
&\hspace{50pt}=(a_1a_4+a_4a_3+a_3a_1)+3(a_1a_2+a_2a_3+a_3a_1)+2(a_1+a_3)a_4\\
&\hspace{50pt}>1+3+0.
\end{align*}

Thus the only possibility is $a_1=a_3$ which implies $\theta_1=\theta_3$. By the similar argument, $0=\frac{\partial G}{\partial \theta_2}$ and $0=\frac{\partial G}{\partial \theta_4}$ imply $\theta_2=\theta_4.$ Since $\theta_1+\theta_2+\theta_3+\theta_4=\pi,$ we have $\theta_1+\theta_2=\frac\pi2$ which implies $a_1a_2=1.$

Now $0=\frac{\partial G}{\partial \theta_1}$ and $0=\frac{\partial G}{\partial \theta_2}$ imply 
$$(3a_2+3a_4+4a_3)(1+a_1^2)=(3a_1+3a_3+4a_4)(1+a_2^2).$$ Since $a_1=a_3$ and $a_2=a_4$, we have 
$$(6a_2+4a_1)(1+a_1^2)=(6a_1+4a_2)(1+a_2^2).$$ Since $a_1a_2=1,$ we have
$$(a_1-a_2)(a_1^2+a_2^2+a_1a_2+1)=0.$$

Since the second factor satisfies $$a_1^2+a_2^2+a_1a_2+1=\frac12(a_1^2+a_2^2)+\frac12(a_1+a_2)^2+1>0,$$ the only possibility is $a_1=a_2.$

Therefore the function $g=3(a_1a_2+a_2a_3+a_3a_4+a_4a_1)+4(a_1a_3+a_2a_4)$ has the unique critical point $(\frac\pi4,\frac\pi4,\frac\pi4,\frac\pi4).$ 

Next, we claim that $g>0$. Since at least three of $a_1,a_2,a_3,a_4$ are positive, without loss of generality, we may assume that $a_1>0,a_2>0,a_3>0.$
Let's write $$g=2(a_1a_2+a_2a_4+a_4a_1)+2(a_2a_3+a_3a_4+a_4a_2)+(a_2+a_4)a_1+(a_1+a_4)a_3+4a_1a_3.$$
Then each term of sum above is positive.

At last, we investigate the behavior of $g$ when the variable approaches the boundary of the domain $\Omega_4$.
Let $(\theta_1(t),\theta_2(t), \theta_3(t), \theta_4(t))$, $t\in [0,\infty),$ be a path in the domain $\Omega_4$. Let $a_i(t)=\cot \theta_1(t)$ for $i=1,2,3,4.$

Without loss of generality, we have
$$\lim_{t\to \infty} (\theta_1(t),\theta_2(t), \theta_3(t), \theta_4(t))=(0, s_2, s_3, s_4)$$ where $s_i\geq 0$ for $i=2,3,4$ and $s_2+s_3+s_4=\pi.$ And we can assume that $s_2<\frac\pi2$ and $s_3<\frac\pi2$. 

Let's write
\begin{align*}
g&=2(a_1(t)a_2(t)+a_2(t)a_4(t)+a_4(t)a_1(t))\\
&+2(a_2(t)a_3(t)+a_3(t)a_4(t)+a_4(t)a_2(t))\\
&+(a_2(t)+a_4(t))a_1(t)+(a_1(t)+a_4(t))a_3(t)+4a_1(t)a_3(t).
\end{align*}
By the inequality (\ref{fm:ine}), $$a_1(t)a_2(t)+a_2(t)a_4(t)+a_4(t)a_1(t)>1$$ and $$a_2(t)a_3(t)+a_3(t)a_4(t)+a_4(t)a_2(t)>1$$ for any $t\in [0, \infty)$.
Since $a_2(t)+a_4(t)>0$ when $t$ is sufficiently large, $\lim_{t\to \infty}(a_2(t)+a_4(t))a_1(t)=\infty.$ And $(a_1(t)+a_4(t))a_3(t)>0, 4a_1(t)a_3(t)>0$ when $t$ is sufficiently large. Hence $g$ approaches $\infty.$

Therefore $g$ has a lower bound and can not achieve its absolute minimum at a boundary point. It much achieve its absolute minimum at the unique critical point $(\frac\pi4,\frac\pi4,\frac\pi4,\frac\pi4).$

\subsection{} We verify the statement about $\lambda_1$ in this subsection. First, we verify that $\lambda_1\leq 2$ as follows.
Let 
\begin{align*}
Q(x):=\frac{P_4(x)}x&=x^3-2(a_1+a_2+a_3+a_4)x^2\\
& \hspace{65pt}+(3(a_1a_2+a_2a_3+a_3a_4+a_4a_1)+4(a_1a_3+a_2a_4))x\\
& \hspace{190pt} -4(a_1+a_2+a_3+a_4).
\end{align*} 

We have $Q(0)=-4(a_1+a_2+a_3+a_4)\leq -16.$ 

If $Q(2)> 0,$ then the first root of $Q(x)$ is less that $2$, i.e., $\lambda_1< 2.$   

If $Q(2)\leq 0$, we claim that $Q'(0)>0$ and $Q'(2)\leq 0.$ Once the two statements are established, $\lambda_1\leq \lambda_2\leq 2.$

In fact $Q'(0)=3(a_1a_2+a_2a_3+a_3a_4+a_4a_1)+4(a_1a_3+a_2a_4)\geq 20.$

To verify $Q'(2)\leq 2,$ we need to use the assumption $Q(2)\leq 2$. In fact $Q(2)\leq 2$ implies 
$$3(a_1a_2+a_2a_3+a_3a_4+a_4a_1)+4(a_1a_3+a_2a_4)\leq 6(a_1+a_2+a_3+a_4)-4.$$ 
Now 
\begin{align*}
&Q'(2)\\
&=12-8(a_1+a_2+a_3+a_4)+3(a_1a_2+a_2a_3+a_3a_4+a_4a_1)+4(a_1a_3+a_2a_4)\\
&\leq 12-8(a_1+a_2+a_3+a_4)+6(a_1+a_2+a_3+a_4)-4\\
&=8-2(a_1+a_2+a_3+a_4)\\
&\leq 0,
\end{align*}
since $a_1+a_2+a_3+a_4\geq 4.$

Second, we verify that $\lambda_1=2$ if and only if $\theta_1=\theta_2=\theta_3=\theta_4=\frac\pi4$. Since $\lambda_1=2$ is the first root of $Q(x)$, we have $Q'(2)\geq 0.$ On the other hand, it is shown that $Q(2)\leq 0$ implies $Q'(2)\leq 0.$ Hence the only possibility is $Q'(2)=0.$ This requires that $a_1+a_2+a_3+a_4=4$. Therefore we must have $\theta_1=\theta_2=\theta_3=\theta_4=\frac\pi4$.

\section{general cyclic polygons}

\subsection{} Assume $n\geq 5.$ The vertices of a cyclic $n$-gon decompose its circumcircle into $n$ arcs. We assume the radius of the circumcircle is 1 and the lengths of the $n$ arcs are $2\theta_1,2\theta_2,...,2\theta_n.$ 

The discrete Laplace operator of a cyclic $n$-gon is independent of the choice of a triangulation. It is 

$$
L_n=\left(
\begin{array}{ccccccc}
a_1+a_n & -a_1    &0        &0          &...   & -a_n \\
-a_1    & a_1+a_2 &-a_2     &0          &...   & 0    \\
0       &-a_2     &a_2+a_3  & -a_3      &...   & 0  \\
0       & 0       &-a_3     &  a_3+a_4  &...   &0 \\
0       &0        &0        & -a_4      &...   &0 \\
  \      &  \       &\    &   \        & \ddots     &\  \\
-a_n      & 0        &0       &0           &...  &a_{n-1}+a_n
\end{array}
\right).
$$

The eigenvalues are $0=\lambda_0\leq \lambda_1 \leq ... \leq \lambda_{n-1}.$

\subsection{} We have $\sum_{i=1}^{n-1}\lambda_i=2\sum_{i=1}^na_i.$ By the similar argument in the case of triangles and cyclic quadrilaterals, we can show that $\sum_{i=1}^na_i$ has the unique critical point $(\theta_1,...,\theta_n)=(\frac\pi n,...,\frac\pi n)$.

Since there is at most one non-positive number in $a_1,...,a_n$, without loss of generality, we may assume $a_1>0,...,a_{n-1}>0.$ Since $a_{n-1}+a_{n}>0,$ we have $\sum_{i=1}^na_i>0.$  

We investigate the behavior of $\sum_{i=1}^na_i$ when the variable approaches the boundary of the domain 
$$\Omega_n=\{(\theta_1,...,\theta_n)\ | \  \theta_1+...+\theta_n=\pi, \theta_i>0, i=1,...,n \}.$$
Let $(\theta_1(t),\theta_2(t), \theta_3(t),\theta(t))$, $t\in [0,\infty),$ be a path in the domain $\Omega_4$. Let $a_i(t)=\cot \theta_1(t)$ for $i=1,2,3,4.$

Without loss of generality, we have
$$\lim_{t\to \infty} (\theta_1(t),...,\theta_n(t))=(0, s_2,..., s_n)$$ where $s_i\geq 0$ for $i=2,...,n$ and $s_2+...+s_n=\pi.$ And we can assume that 
$s_2<\frac\pi2,...,s_{n-1}<\frac\pi2.$ Since $a_i(t)>0$ for $i=2,...,n-1$ and $a_{n-1}+a_n>0$ when $t$ is sufficiently large, $\lim_{t\to \infty}a_1=\infty$ implies that  $\lim_{t\to \infty} \sum_{i=1}^na_i=\infty.$

Thus $\sum_{i=1}^na_i$ achieved the absolute minimum at $(\frac\pi n,...,\frac\pi n)$.

\subsection{} In this subsection we verify the statement about $\prod_{i=1}^{n-1}\lambda_i.$

\begin{2}
Let $M$ be an $n$ by $n$ matrix. If the sum of the entries of each row or each column of $M$ vanishes, all principle $n-1$ by $n-1$ submatrices of $M$ have the same determinant, and this value is equal to $\frac 1n$ times the product of all nonzero eigenvalues of $M$.
\end{2}

For the reference of the weighted matrix-tree Theorem, for example, see \cite{LW}, page 450, Problem 34A or \cite{DKM}, Theorem 1.2.

In our case, according the weighted matrix-tree Theorem, to calculate $\prod_{i=1}^{n-1}\lambda_i$ of the matrix $L_n$, it is enough to calculate a particular principle $n-1$ by $n-1$ matrix. 
\begin{lemma}
Let $N_n$ be the submatrix obtained by deleting the first row and first column of the matrix $L_n$. Then
$$\det N_n=\sum_{i=1}^n a_1...\widehat{a_{i}}...a_n,$$
where $\widehat{a_{i}}$ means that $a_i$ is missing.
\end{lemma}
\begin{proof} We prove the statement by the mathematical induction. It holds for $n=4$ as we see in the formula of the characteristic polynomial $P_4(x)$. We assume it holds for $n\leq m-1.$ By the property of tridiagonal matrices, we have
$$\det N_m=(a_{m-1}+a_m)\det N_{m-1}-a_{m-1}^2\det N_{m-2}.$$ Then by the assumption of the induction,
\begin{align*}
\det N_m&=(a_{m-1}+a_m)\sum_{i=1}^{m-1} a_1...\widehat{a_{i}}...a_{m-1}-a_{m-1}^2\sum_{i=1}^{m-2} a_1...\widehat{a_{i}}...a_{m-2}\\
&=a_{m-1}\sum_{i=1}^{m-1} a_1...\widehat{a_{i}}...a_{m-1}-a_{m-1}\sum_{i=1}^{m-2} a_1...\widehat{a_{i}}...a_{m-2}a_{m-1}\\
&\hspace{200pt}+a_m\sum_{i=1}^{m-1} a_1...\widehat{a_{i}}...a_{m-1}\\
&=a_1...a_{m-1}+a_m\sum_{i=1}^{m-1} a_1...\widehat{a_{i}}...a_{m-1}\\
&=\sum_{i=1}^m a_1...\widehat{a_{i}}...a_m.\\
\end{align*}

\end{proof}

In the following, we prove that $\sum_{i=1}^n a_1...\widehat{a_{i}}...a_n$ achieves its absolute minimum when $\theta_1=...=\theta_n=\frac\pi n$. 
It is enough to show that 
\begin{itemize}
\item[a.] $\sum_{i=1}^n a_1...\widehat{a_{i}}...a_n$ has the unique critical point $(\frac\pi n,...,\frac\pi n)$;
\item[b.] $\sum_{i=1}^n a_1...\widehat{a_{i}}...a_n>0$;
\item[c.] $\sum_{i=1}^n a_1...\widehat{a_{i}}...a_n$ approaches $\infty$ as the variable approaches the boundary of the domain $\Omega_n$.
\end{itemize}

When $n=4,$ since $a_1a_2a_3+a_1a_2a_4+a_1a_3a_4+a_2a_3a_4=a_1+a_2+a_3+a_4,$ the three statements above are already shown to be true in section 3.
We assume that the three statements above hold when $n\leq m-1.$ 

Let's check the three statements when $n=m.$ Consider the function
$$H=\sum_{i=1}^m a_1...\widehat{a_{i}}...a_m-y(\theta_1+...\theta_m-\pi).$$
Then $0=\frac{\partial H}{\partial \theta_1}$ and $0=\frac{\partial H}{\partial \theta_2}$ imply that 
$$(a_3...a_m+a_2\sum_{i=3}^m a_3...\widehat{a_{i}}...a_m)(1+a_1^2)=(a_3...a_m+a_1\sum_{i=3}^m a_3...\widehat{a_{i}}...a_m)(1+a_2^2).$$
Since $a_1+a_2>0,$ it is equivalent to
$$(a_1-a_2)(a_1+a_2)(a_3...a_m+\frac{a_1a_2-1}{a_1+a_2}\sum_{i=3}^m a_3...\widehat{a_{i}}...a_m)=0.$$ 

The third factor is $$\cot\theta_3...\cot\theta_m+\cot(\theta_1+\theta_2)\sum_{i=3}^m\cot\theta_3...\widehat{\cot\theta_i}...\cot\theta_m$$ which is written as $\sum_{i=1}^m \widetilde{a}_1...\widehat{\widetilde{a}_i}...\widetilde{a}_{m-1}$, where $\widetilde{a_1}=\cot(\theta_1+\theta_2), \widetilde{a}_i=\cot\theta_{i+1}$ for $i=2,...,m-1.$ This expression corresponds to a cyclic $(m-1)$-gon. By assumption of the induction, $\sum_{i=1}^m \widetilde{a}_1...\widehat{\widetilde{a}_i}...\widetilde{a}_{m-1}>0.$

Hence the only possibility is $a_1=a_2.$ By similar argument, we show that $a_i=a_j$ for any $i,j.$ Hence the function $\sum_{i=1}^m a_1...\widehat{a_i}...a_m$ has the unique critical point such that $\theta_i=\frac\pi m$ for any $i=1,...,m.$

Next, we claim that $\sum_{i=1}^m a_1...\widehat{a_i}...a_m>0.$  
Without loss of generality, we assume that $a_1>0, a_2>0, ..., a_{m-1}>0.$ Now
\begin{align*}
&\sum_{i=1}^m a_1...\widehat{a_i}...a_m\\
&=a_1a_2...a_{m-2}(a_{m-1}+a_m)+\sum_{i=1}^{m-2} a_1...\widehat{a_i}...a_{m-2}(a_{m-1}a_m)\\
&=a_1a_2...a_{m-2}(a_{m-1}+a_m)+\sum_{i=1}^{m-2} a_1...\widehat{a_i}...a_{m-2}(a_{m-1}a_m-1)+\sum_{i=1}^{m-2} a_1...\widehat{a_i}...a_{m-2}\\
&=(a_{m-1}+a_m)(a_1a_2...a_{m-2}+\sum_{i=1}^{m-2} a_1...\widehat{a_i}...a_{m-2}\frac{a_{m-1}a_m-1}{a_{m-1}+a_m})+\sum_{i=1}^{m-2} a_1...\widehat{a_i}...a_{m-2}.
\end{align*}

Let $\widetilde{a}_{m-1}=\frac{a_{m-1}a_m-1}{a_{m-1}+a_m}=\cot(\theta_{m-1}+\theta_m).$ Then 
\begin{align*}
&a_1a_2...a_{m-2}+\sum_{i=1}^{m-2} a_1...\widehat{a_i}...a_{m-2}\frac{a_{m-1}a_m-1}{a_{m-1}+a_m}\\
&=a_1a_2...a_{m-2}+\sum_{i=1}^{m-2} a_1...\widehat{a_i}...a_{m-2}\widetilde{a}_{m-1}.\\
\end{align*}
Consider an cyclic $(m-1)$-gon with angles $\theta_1,...,\theta_{m-2}, \theta_{m-1}+\theta_m.$ By the assumption of induction, $$a_1a_2...a_{m-2}+\sum_{i=1}^{m-2} a_1...\widehat{a_i}...a_{m-2}\widetilde{a}_{m-1}>0.$$ Therefore $\sum_{i=1}^m a_1...\widehat{a_i}...a_m>0.$

At last, we investigate the behavior of the function $\sum_{i=1}^m a_1...\widehat{a_i}...a_m$ when the variable approaches the boundary of the domain $$\Omega_m=\{(\theta_1,...,\theta_m)\ | \  \theta_1+...+\theta_m=\pi, \theta_i>0, i=1,...,m \}.$$ Let $(\theta_1(t),..., \theta_m(t))$, $t\in [0,\infty),$ be a path in the domain $\Omega_m$. Let $a_i(t)=\cot \theta_1(t)$ for $i=1,...,m.$

Without loss of generality, we assume $$\lim_{t\to \infty}(\theta_1(t),...,\theta_m(t))=(0,s_2,...,s_m),$$ where $s_2\geq 0,...,s_m \geq 0$ and $s_2+...+s_m=\pi.$ And we can assume furthermore that $s_2<\frac\pi2,...,s_{m-1}<\frac\pi2.$ Thus $a_1(t)>0,...,a_{m-1}(t)>0$ when $t$ is sufficiently large. To simplify the notation, we denote $a_i(t)$ by $a_i$ in the follows.
Now 
\begin{align*}
&\sum_{i=1}^m a_1...\widehat{a_i}...a_m\\
&=(a_{m-1}+a_m)(a_1a_2...a_{m-2}+\sum_{i=1}^{m-2} a_1...\widehat{a_i}...a_{m-2}\frac{a_{m-1}a_m-1}{a_{m-1}+a_m})+\sum_{i=1}^{m-2} a_1...\widehat{a_i}...a_{m-2}\\
&=(a_{m-1}+a_m)(a_1a_2...a_{m-2}+\sum_{i=1}^{m-2} a_1...\widehat{a_i}...a_{m-2}\widetilde{a}_{m-1})+\sum_{i=1}^{m-2} a_1...\widehat{a_i}...a_{m-2},
\end{align*}
where $\widetilde{a}_{m-1}=\frac{a_{m-1}a_m-1}{a_{m-1}+a_m}=\cot(\theta_{m-1}+\theta_m).$

By the assumption of induction, $$a_1a_2...a_{m-2}+\sum_{i=1}^{m-2} a_1...\widehat{a_i}...a_{m-2}\widetilde{a}_{m-1}>0$$ for any $t\in [0,\infty)$. Since $a_{m-1}+a_m>0$ for any $t\in [0,\infty)$ and $\sum_{i=1}^{m-2} a_1...\widehat{a_i}...a_{m-2}$ approaches $\infty,$ we see that $\sum_{i=1}^m a_1...\widehat{a_i}...a_m$ approaches $\infty.$

\end{document}